\documentclass[reqno]{amsart}
\usepackage{mathtools}
\usepackage{graphicx}
\usepackage[unicode]{hyperref}

\newtheorem{theorem}{Theorem}
\newtheorem{lemma}{Lemma}
\newtheorem{prop}{Proposition}
\numberwithin{equation}{section}

\newcommand{\R}{\mathbb{R}}
\newcommand{\N}{\mathbb{N}}
\newcommand{\D}{\mathsf{D}}
\newcommand*{\x}{\mathrel{\scalebox{0.6}{$\times$}}}
\newcommand{\Sym}{\R^{d\x d}_{\rm sym}}
\newcommand{\PD}{\R^{d\x d}_{+}}
\newcommand{\dd}[1]{\,\mathrm{d}{#1}}
\newcommand{\Wi}{W^{1,\infty}_{\rm loc}}
\newcommand{\Win}{W^{1,\infty}_{\rm loc}(V;\PD)}

\begin{document}

\title[Gradient coercivity of SPD fields]{Sharp nonlinear estimates for multiplying derivatives of positive definite tensor fields}
\author[M. Bathory]{Michal Bathory}
\address{Michal Bathory, University of Vienna, Faculty of Mathematics, Oskar-Morgernstern-Platz~1, 1090 Wien, Austria}
\email{michal.bathory@univie.ac.at}
\thanks{Submitted to Mathematical Inequalities \& Applications, \url{http://mia.ele-math.com/}}
\hypersetup{pdftitle=Optimal inequalities in multiplication of derivatives of positive definite matrices and their powers,
	pdfauthor=Michal Bathory,
	hidelinks,
	pdfstartview=FitH
}

\begin{abstract}
  The simple product formulae for derivatives of scalar functions raised to different powers are generalized for functions which take values in the set of symmetric positive definite matrices. These formulae are fundamental in derivation of various non-linear estimates, especially in the PDE theory. To get around the non-commutativity of the matrix and its derivative, we apply some well-known integral representation formulas and then we make an observation that the derivative of a matrix power is a logarithmically convex function with respect to the exponent. This is directly related to the validity of a seemingly simple inequality combining the integral averages and the inner product on matrices. The optimality of our results is illustrated on numerous examples.
\end{abstract}

\maketitle

\section{Introduction and main results}

Let $V\subset\R^n$, $n\in\N$, be an open set and let $\D_i$, $i=1,\ldots,n$, denote the partial derivatives. As a consequence of the differentiation rules for the real power function $x\mapsto x^{\alpha}$, $\alpha\neq-1$, the identities
\begin{equation}\label{gen}
\D_i u\,u^{\alpha}=\frac1{\alpha+1}\D_i u^{\alpha+1}\quad\text{and}\quad
\sum_{i=1}^n\D_i u\,\D_i u^{\alpha}=\frac{4\alpha}{(\alpha+1)^2}\sum_{i=1}^n\big|\D_i u^{\frac{\alpha+1}2}\big|^2
\end{equation}
hold true almost everywhere in $V$ for any positive and locally Lipschitz continuous function $u\colon V\to\R_+$ (denoted by $u\in\Wi(V;\R_+)$, see Section~\ref{SecNot} for details). These identities are frequently used in the theory of nonlinear partial differential equations (PDE) to find information about the unknown function. Our goal is to prove \eqref{gen} when the scalar $u$ is replaced by $A\in\Win$, where $\PD$ denotes the set of symmetric positive definite matrices of the size $d\times d$, $d\in\N$. Such a situation occurs in numerous physical applications, see Section~\ref{SecMot} for more details. It turns out that while $\eqref{gen}_1$ continues to hold, the identity $\eqref{gen}_2$ fails due to non-commutativity of $A$ and $\D_i A$. Nevertheless, we show that $\eqref{gen}_2$ can still be recovered as an inequality in the preferable direction. Since our result is more general, let us first define, for any $A\in\Win$ and $\lambda\in\R$, the non-linear differential operator $\D^{\lambda}=(\D^{\lambda}_i)_{i=1}^n$, by
\begin{equation}\label{def}
\D_i^{\lambda}A\coloneqq\Big\{\,\begin{matrix}
\lambda^{-1}\D_i A^{\lambda}&\text{if}\quad\lambda\neq0\\[0.1cm]
\D_i\log A&\text{if}\quad\lambda=0
\end{matrix}\,\Big\}=\int_0^1A^{\lambda(1-s)}(\D_i\log A)A^{\lambda s}\dd{s},
\end{equation}
where we used the matrix power and matrix logarithm functions (see Section~\ref{SecNot}) and the last equality follows from standard results, see Lemma~\ref{Lrepr} below. The case $\lambda=1$ recovers the usual partial derivative $\D_i A=\D_i^1A$. We also denote the Euclidian inner product and norm on the spaces $\R^{m_1\x \ldots\x m_k}$, $k\in\N$, of rank-$k$ tensors by
\begin{equation}\label{Frob}
\langle X,Y\rangle_{m_1\x \ldots\x m_k}\!\!\coloneqq\!\!\sum_{i_1=1}^{m_1}\!\!\ldots\!\!\sum_{i_k=1}^{m_k}X_{i_1\ldots i_k}Y_{i_1\ldots i_k},\; |X|_{m_1\x \ldots\x m_k}\!\!\coloneqq\!\!\sqrt{\langle X,X\rangle_{m_1\x \ldots\x m_k}}
\end{equation}
for any $X,Y\in\R^{m_1\x \ldots\x m_k}$. Then, our generalization of \eqref{gen} takes the following form.

\begin{theorem}\label{Todh}
	Let $p,q\geq0$, $\alpha,\beta,\gamma,\delta\in\R$ and let $A\in\Win$. Then
	\begin{equation}\label{gengen}
	\langle\D^{\alpha}A,A^{\beta}\rangle_{d\x d}=\langle\D^{\alpha+\beta}A,I\rangle_{d\x d},
	\end{equation}
	where $I$ is the identity matrix, and
	\begin{equation}\label{ide0A}
	\langle\D^{\alpha}A,\D^{\beta}A\rangle_{n\x d\x d}^p\langle\D^{\gamma}A,\D^{\delta}A\rangle_{n\x d\x d}^q\geq\Big|\D^\frac{(\alpha+\beta) p+(\gamma +\delta) q}{2p+2q}A\Big|_{n\x d\x d}^{2p+2q}
	\end{equation}
	almost everywhere in $V$.
\end{theorem}

In Section~\ref{SubIne}, we show that \eqref{ide0A} \textit{can not hold with the equality sign} in general. Also, we would like to point out that \textit{the matrix symmetry assumption is important} and that \eqref{ide0A} can not hold (in general) for non-symmetric positive definite valued functions, see Section~\ref{SubSym}. From the analytic point of view, the direction of inequality \eqref{ide0A} is the preferred one as the right hand side is non-negative and has a simple structure. Nevertheless, to provide a~more complete picture, we investigate also the reverse inequality to \eqref{ide0A} in Section~\ref{SecRev}. Using \eqref{def}, we will show that \eqref{ide0A} is rather a simple consequence of the following theorem, which is thus our key result.

\begin{theorem}\label{TT}
	Let $A\in\Win$. Then, the function
	\begin{equation}\label{func}
	\lambda\mapsto\big|\D^{\lambda}A\big|_{n\x d\x d},\quad\lambda\in\R,
	\end{equation}
	is logarithmically convex in the following {\rm(}strengthened\,{\rm)} sense:
	
	For every $\alpha,\beta\in\R$, there holds
	\begin{equation}\label{cox}
	\langle\D^{\alpha}A,\D^{\beta}A\rangle_{n\x d\x d}\geq\big|\D^{\frac{\alpha+\beta}2}A\big|^2_{n\x d\x d}\quad\text{a.e.\ in }V.
	\end{equation}
\end{theorem}

Theorem~\ref{TT} naturally generalizes the scalar case $d=1$, $u\in\Wi(V;\R_+)$, where
\begin{equation*}
\log|\D^{\lambda}u|_n=\lambda\log u+\log|\D\log u|_n,\quad\D u\neq 0,
\end{equation*}
is simply a linear function of $\lambda$. Note that \eqref{cox} takes into account the structure of the Frobenius inner product, unlike the usual definitions of logarithmic convexity which use only the standard multiplication. We remark that \eqref{cox} implies the logarithmic convexity of \eqref{func} in the usual sense, see Lemma~\ref{Llogcon} below.

It turns out that the heart of the matter is the following inequality.

\begin{lemma}\label{Lemf}
	Let $B\in\PD$ and $X\in\R^{d\x d}$, $d\in\N$. Then the function
	\begin{equation}\label{P}
	P(x)\coloneqq\int_0^{1}B^{(1+x)s}XB^{-(1+x)s}\dd{s},\quad x\in\R,
	\end{equation}
	satisfies the inequality
	\begin{equation}\label{Odh}
	\langle P(x),P(-x)\rangle_{d\x d}\geq|P(0)|_{d\x d}^2\quad\text{for all}\quad x\in\R.
	\end{equation}
\end{lemma}

We remark that if $x=\pm1$, then \eqref{Odh} becomes a Jensen inequality for~$|\cdot|_{d\x d}^2$.

In Theorems~\ref{Todh} and \ref{TT} above, the assumption of local Lipschitz continuity is considered for convenience since $\Win$ is a convex cone that is also closed under the operation $A\mapsto A^{\alpha}$, $\alpha\in\R$ (see Lemma~\ref{Lrepr}). At the same time, this setting seems sufficient for many PDE applications. Our results, of course, continue to hold in any subset of $\Win$ (such as $C^k(V;\PD)$, $k\in\N\cup\{\infty\}$, or $W^{1,\infty}(V;\PD)$), but it may no longer be true that $A^{-1}$ belongs to the same set as $A$. On the other hand, in the last Section~\ref{SecLast}, we briefly discuss a~possible relaxation of the assumption $A\in\Win$.

\section{Notation}\label{SecNot}

The set $\R^{d\x d}_{\rm sym}$, $d\in\N$, consists of all symmetric matrices $A\in\R^{d\x d}$, i.e., those which fulfil $A=A^T$, where $A^T$ is the transpose of $A$. Furthermore, the set of all symmetric positive definite matrices $\PD$ consists of all $A\in\Sym$ with the property
\begin{equation}\label{pd}
\langle Av,v\rangle_d>0\quad\text{for all }0\neq v\in\R^d.
\end{equation}
In the special case $d=1$, we abbreviate $\R_+\coloneqq\R^{1\x 1}_+=(0,\infty)$. Seeing the matrix multiplication as a composition of linear operators and the matrix transpose as the operator adjoint, it is not surprising that the identity
\begin{equation}
\langle A_1A_2A_3A_4, I\rangle_{d\x d}\!=\!\langle A_1A_2A_3,A_4^T\rangle_{d\x d}\!=\!\langle A_1A_2,A_4^TA_3^T\rangle_{d\x d}\!=\!\langle A_2,A_1^TA_4^TA_3^T\rangle_{d\x d}\label{I2}
\end{equation}
holds for all $A_1,A_2,A_3,A_4\in\R^{d\x d}$, where $I$ is always the identity matrix. Therefore, since for any $B_1,B_2\in\PD$, we can write $B_1=B^{\frac12}_1B^{\frac12}_1$ and $B_2=B_2^{\frac12}B_2^{\frac12}$ (see below), we obtain, for all $A\in\R^{d\x d}$, that
\begin{equation}
\langle B_1AB_2,A\rangle_{d\x d}=\big|B_1^{\frac12}A B_2^{\frac12}\big|_{d\x d}^2.\label{I1}
\end{equation}

As a~consequence of the~Schur decomposition, every symmetric (and thus normal) matrix $A$ admits a~spectral decomposition of the~form
\begin{equation}\label{rozkl}
A=Q D Q^T,
\end{equation}
where $D$ is a~diagonal matrix containing the~real eigenvalues $\{\lambda_i\}_{i=1}^d$ of $A$ and $Q$ is a~unitary matrix of the corresponding eigenvectors, see \cite[p.~101]{Horn1990} or \cite[p.~17]{Magnus1999}. Then, we can extend the definition of any real function $f\colon\R\to\R$ to symmetric matrix arguments via
\begin{equation}\label{pow}
f(A)=Qf(D)Q^T,\quad A\in \Sym,
\end{equation}
where $f(D)$ is diagonal matrix with entries $f(D_{ii})$, $i=1,\ldots,d$ on its diagonal. If the natural domain of the function $f$ is $\R_+$ (such as for $x^{\alpha}$ or $\log x$), we can still use \eqref{pow} to define $f(A)$ provided that $A\in\PD$. Using definition \eqref{pow}, it is easy to see that all the basic calculus identities remain true also in the matrix case, for example:
\begin{equation}
A^{\alpha}A^{\beta}=A^{\alpha+\beta}=A^{\beta}A^{\alpha},\quad\log A^{\alpha}=\alpha\log A,\quad\exp\log A=A,\qquad\alpha,\beta\in\R.\label{explog}
\end{equation}

The symbol $V$ always denotes an open subset of $\R^n$, $n\in\N$. Let $N\in\N$, $1\leq p\leq \infty$ and let us recall that the Sobolev space $W^{1,p}(V;\R^N)$ is defined as the set of all functions $u\colon V\to\R^N$ whose distributional gradient can be represented by a locally integrable function $\D u$ and the norm
\begin{equation*}
\| u\|_{W^{1,p}(V;\R^N)}\coloneqq\Big\{\begin{matrix}
\left(\int_V(|u|_N^p+|\D u|_{n\x N}^p)\right)^{\frac1p}&\quad\text{if}\quad p<\infty;\\[0.1cm]
\;\underset{V}{\text{ess\,sup}}\,(|u|_N+|\D u|_{n\x N})&\quad\text{if}\quad p=\infty\;
\end{matrix}
\end{equation*}
is finite. The space $W^{1,p}(V;\R^{d\x d})$ is then defined analogously. We refer to \cite{Adams2003} for properties of Sobolev spaces. We define the set $\Win$ as
\begin{equation}\label{def1}
\big\{\,A\colon V\to\PD:\|A\|_{W^{1,\infty}(K;\R^{d\x d})}<\infty\quad\text{for all}\;\; K\;\;\text{open with}\;\;\overline{K}\subset V\,\big\}.
\end{equation}
Although this is not a vector space (it is not closed under subtraction), it has other useful properties (most importantly, it is invariant with respect to the matrix inverse) as is shown in Lemma~\ref{Lrepr} below. It is known that functions from $W^{1,\infty}(\R^n;\R^{d\x d})$ are continuous (up to a null set) and in fact as a consequence of Morrey's inequalities, it is not hard to see that $\Wi(V;\R^{d\x d})$ coincides with the traditional space $C^{0,1}_{\rm loc}(V;\R^{d\x d})$ of locally Lipschitz functions, whose classical derivative exists a.e.\ in $V$. Nevertheless, we stick to the notation $\Wi(V;\R^{d\x d})$ (and $\Win$), since the definition \eqref{def1} is easy to work with in what follows.

\section{PDE motivation and related results}\label{SecMot}

Our motivation to investigate \eqref{ide0A} originates from the study of certain non-linear partial differential equations arising in the theory of viscoelastic fluids. These equations contain a tensorial function as an unknown and they have been used by physicist and engineers for a~long time, see e.g.~\cite{Oldroyd} or \cite{Bird}. We refrain from introducing these complex equations in detail here. Instead, we shall present here only an illustrational example involving a nonlinear Poisson equation with Dirichlet boundary conditions. This example nicely demonstrates how \eqref{ide0A} can be applied to improve information about the solution. 

\subsection{Application of (1.5)} Suppose that $\Omega\subset\R^3$ is an open set with a Lipschitz boundary. Then, let us consider the boundary value problem
\begin{align}\label{syst1}
-\sum_{i=1}^3\D_i\big(\sqrt{|\D A|_{3\x d\x d}}\,\D_iA\big)&=F\quad\text{in}\quad\Omega,\\
A&=0\quad\,\text{on}\quad \partial\Omega,\label{syst2}
\end{align}
for an unknown function $A$ and with the data satisfying $F\in L^5(\Omega;\Sym)$ (i.e., $|F|_{d\x d}^5$ is Lebesgue-integrable in $\Omega$). Then, we claim that every (distributional) positive definite solution of \eqref{syst1} and \eqref{syst2} must actually satisfy
\begin{equation}\label{prop}
\Big(\int_{\Omega}|A|_{d\x d}^{45}\Big)^{\frac16}+\int_{\Omega}|\D A^3|_{3\x d\x d}^{\frac52}\leq C\int_{\Omega}|F|^5_{d\x d}<\infty
\end{equation}
with some $C>0$ depending only on $\Omega$. This can be seen by taking the inner product of both sides of \eqref{syst1} with $A^6$, integrating the result over $\Omega$, integrating by parts and using \eqref{syst2}, leading to
\begin{equation}\label{eqa}
\int_{\Omega}\sqrt{|\D A|_{3\x d\x d}}\langle\D A,\D A^6\rangle_{3\x d\x d}=\int_{\Omega}\langle F,A^6\rangle_{d\x d}.
\end{equation}
To get any useful information out of this, one would like to proceed as in the scalar case: estimate the integrand on the left from below by a simpler expression of the type $|\D A^{\lambda}|_{3\x d\x d}^r$. In the matrix case, this seems not so easy. Nevertheless, a straightforward application of Theorem~\ref{Todh} with $p=\frac14$, $\alpha=\beta=1$ and $q=1$, $\gamma=1$, $\delta=6$, gives
\begin{equation}\label{po}
\sqrt{|\D A|_{3\x d\x d}}\langle\D A,\D A^6\rangle_{3\x d\x d}\geq\frac{2}{\sqrt{27}}|\D A^3|_{3\x d\x d}^{\frac52}.
\end{equation}
Then, we use the Young inequality and the Cauchy-Schwarz inequality to estimate
\begin{equation}\label{intest}
\int_{\Omega}\langle F,A^6\rangle_{d\x d}\leq\varepsilon\int_{\Omega}|A^6|_{d\x d}^{\frac54}+C(\varepsilon)\int_{\Omega}|F|_{d\x d}^5\leq\varepsilon\int_{\Omega}|A^3|_{d\x d}^{\frac52}+C(\varepsilon)\int_{\Omega}|F|_{d\x d}^5.
\end{equation}
If we use this inequality with $\varepsilon>0$ sufficiently small together with \eqref{po} in \eqref{eqa} and apply the Poincar\'e inequality, we obtain
\begin{equation}\label{poinc}
\int_{\Omega}|\D A^3|_{3\x d\x d}^{\frac52}\leq C\int_{\Omega}|F|^5_{d\x d}.
\end{equation}
Then, using the Sobolev embedding $W^{1,\frac52}(\Omega;\Sym)\hookrightarrow L^{15}(\Omega;\Sym)$ and also the inequality
\begin{equation}\label{str2}
|A|^{45}_{d\x d}\leq3^{15}|A^3|^{15}_{d\x d}
\end{equation}
(explained below), we arrive at \eqref{prop}. 

We would like to point out that although one could test also by the functions of the type $\langle A,I\rangle^{\lambda}_{d\x d}I$ (where the result is easier to manipulate), this can never yield the optimal gradient estimate \eqref{poinc}.

\subsection{Matrix power and matrix norm ``commute''}

Note that in the example above, inequality \eqref{str2} was also quite important (besides \eqref{po}). While \eqref{str2} may seem obvious after a while, this may be not be the case for similar inequalities with different natural, rational or even real exponents. However, due to the next proposition, we can manipulate the powers and norms of positive definite matrices analogously as in the scalar case, with certain multiplicative constants and up to one exception.

\begin{prop}\label{Tests}
	Let $A\in\PD$. Then the following estimates hold:
	\begin{align}
	|A|_{d\x d}&\leq \langle A,I\rangle_{d\x d}\leq\sqrt{d}|A|_{d\x d};\label{I7}\\
	\min\{1,d^{\frac{1-\alpha}2}\}|A|_{d\x d}^{\alpha}&\leq|A^{\alpha}|_{d\x d}\leq\max\{1,d^{\frac{1-\alpha}2}\}|A|_{d\x d}^{\alpha}\;\;\text{for any}\;\;\alpha\geq0;\label{I8}\\
	\min\{d^{\frac12},d^{\frac{-\alpha}2}\}|A|_{d\x d}^{\alpha}&\leq|A^{\alpha}|_{d\x d}\hskip3.7cm\text{for any}\;\;\alpha\leq 0.\label{I9}
	\end{align}
\end{prop}
\begin{proof}
	By the Cauchy-Schwarz inequality (in $\R^d$ and then in $\R^{d\x d}$) and Young's inequality, we get
	\begin{equation*}
	|A|_{d\x d}=\big|A^{\frac12}A^{\frac12}\big|_{d\x d}\leq\big|A^{\frac12}\big|_{d\x d}^2=\langle A,I\rangle_{d\x d}\leq|A|_{d\x d}|I|_{d\x d}=\sqrt{d}|A|_{d\x d},
	\end{equation*}
	which proves \eqref{I7}.
	
	Next, for $\alpha\in[0,\infty)$, we denote $\sigma(\alpha)=\sum_{i=1}^d\lambda_i^{2\alpha}$, where $\lambda_i=D_{ii}$ and $D$ is defined in \eqref{rozkl}. If we use concavity of the~power function $x\mapsto x^{\alpha}$ for $\alpha\in[0,1]$ twice (first time in the~form $\varepsilon x^{\alpha}\leq (\varepsilon x)^{\alpha}$, $\varepsilon\in(0,1)$), we get the~inequality
	\begin{equation*}
	\sigma(1)^{\alpha}\!=\sum_{i=1}^d\frac{\lambda_i^2}{\sigma(1)}\sigma(1)^{\alpha}\!\leq\sum_{i=1}^d\lambda_i^{2\alpha}\!
	=\sigma(\alpha)=d\sum_{i=1}^d\frac{\lambda_i^{2\alpha}}{d}\leq d\bigg(\sum_{i=1}^d\frac{\lambda_i^2}d\bigg)^{\alpha}=d^{1-\alpha}\sigma(1)^{\alpha}.
	\end{equation*}
	Thus, since $\sigma(\alpha)^{\frac12}=|D^{\alpha}|_{d\x d}=|A^{\alpha}|_{d\x d}$, we obtain
	\begin{equation}\label{conc}
	|A|_{d\x d}^{\alpha}=\sigma(1)^{\frac{\alpha}2}\leq \sigma(\alpha)^{\frac12}=|A^{\alpha}|_{d\x d}=\sigma(\alpha)^{\frac12}\leq d^{\frac{1-\alpha}2}\sigma(1)^{\frac{\alpha}2}=d^{\frac{1-\alpha}2}|A|_{d\x d}^{\alpha}.
	\end{equation}
	Analogously, for $\alpha\in[1,\infty)$, the convexity of $x\mapsto x^{\alpha}$ leads to
	\begin{equation}\label{conv}
	d^{\frac{1-\alpha}2}|A|_{d\x d}^{\alpha}\leq|A^{\alpha}|_{d\x d}\leq |A|_{d\x d}^{\alpha},
	\end{equation}
	which finishes the proof of \eqref{I8}. 
	
	To prove \eqref{I9}, note first that
	\begin{equation*}
	\sqrt{d}=|I|_{d\x d}=|BB^{-1}|_{d\x d}\leq|B|_{d\x d}|B^{-1}|_{d\x d}\quad\text{for any}\quad B\in\PD.
	\end{equation*}
	Hence, on choosing $B=A^{\alpha}$ and using the second inequality in \eqref{I8}, we get
	\begin{equation*}\big|A^{\alpha}\big|_{d\x d}\geq\sqrt{d}\big|A^{-\alpha}\big|_{d\x d}^{-1}\geq \sqrt{d}\max\{1,d^{\frac{1+\alpha}2}\}^{-1}|A|_{d\x d}^{\alpha}=\min\{d^{\frac12},d^{-\frac{\alpha}2}\}|A|_{d\x d}^{\alpha}
	\end{equation*}
	and the proof of the theorem is finished.
\end{proof}

The missing upper bound in \eqref{I9} can not hold as can be seen by considering, for example, the case $\alpha=-1$ and matrices of the form
\begin{equation*}
A_0=\left(\begin{matrix}
k&0\\0&k^{-1}
\end{matrix}\right),\quad k\in\N.
\end{equation*}


\section{Optimality, (counter-)examples and the reverse inequality}\label{SecOpt}

In this section, using only simple arguments and examples, we argue that the assumptions and conclusions of Theorem~\ref{Todh} are optimal in several aspects. To this end, we will explicitly evaluate both sides of \eqref{ide0A} in the case $d=2$, $n=1$, $\alpha=1$, $\beta=3$, $p=1$, $q=0$, for appropriately chosen functions $A_i\in\Wi(\R;\R^{2\x 2}_+)$. This case seems ideal as it is particularly easy to evaluate in hand, while exhibiting fully the non-commutativity of $A$ and $\D A$. It should be intuitively clear that the examples below and their analogies would work also for the other choices of the parameters $d\geq 2$, $n$, $\alpha$, $\beta$, $\gamma$, $\delta$, $p$, $q$, however proving this rigorously would be too exhaustive. Thus, the examples and conclusions in this section should be perceived only as strong indications of optimality of Theorem~\ref{Todh} (and its converse), but nothing more.

\subsection{Why (1.5) is only an inequality?}\label{SubIne}

Theorem~\ref{Todh} implies that
\begin{equation}\label{inee}
\langle\D A,\D A^3\rangle_{2\x2}\geq\frac34\big|\D A^2\big|_{2\x2}^2
\end{equation}
and we do not hope to improve the factor $\frac34$ (for general $A$) since \eqref{inee} is always an equality if $d=1$. However, we may still ask why \eqref{inee} is only an inequality when $d>1$. To answer this, let us consider the function
\begin{equation*}
A_1(x)=\left(\begin{matrix}
\cosh x&1\\
1&2
\end{matrix}\right),\quad x\in\R.
\end{equation*}
As $\cosh{x}\geq1$ for every $x\in\R$, the matrix $A_1$ is positive definite in $\R$. Note also, that $A_1$ and $\D_1 A_1$ commute only if $x=0$. Then, denoting
\begin{equation}\label{rdef}
r_{A}(x)\coloneqq\frac{\langle\D A(x),\D A^3(x)\rangle_{2\x2}}{|\D A^2(x)|_{2\x2}^2},\quad x\in\R,
\end{equation}
and performing some elementary algebra, we discover that
\begin{equation*}
r_{A_1}(x)=\frac{(\sinh{x})(2\sinh{x} + 3\cosh^2{x} \sinh{x})}{\left|\left(\begin{matrix}
	2 \cosh{x}\sinh{x} & \sinh{x}\\
	\sinh{x} &0
	\end{matrix}\right)\right|_{2\x2}^2}=\frac34+\frac{1}{4+8\cosh^2{x}},
\end{equation*}
which is always strictly greater than $\frac34$. This shows that we can not expect \eqref{inee} to hold with equality, unlike in the scalar case, where $r$ always evaluates to $\frac34$, of course. We support this claim by another, this time only graphical example: see Figure~1 for the graph of the function $r_{A_2}$, where
\begin{equation}\label{osc}
A_2(x)=\left(\begin{matrix}
\cosh x&\tfrac15\sin(5x)\\
\tfrac15\sin(5x)&1
\end{matrix}\right),\quad x\in\R.
\end{equation}
There we can see nicely that $\frac34$ is indeed an optimal lower bound in \eqref{inee} (and that this remains true even if we restrict $x$ to a smaller domain).
\begin{figure}\label{figg}
	\centering
	\includegraphics[scale=0.4]{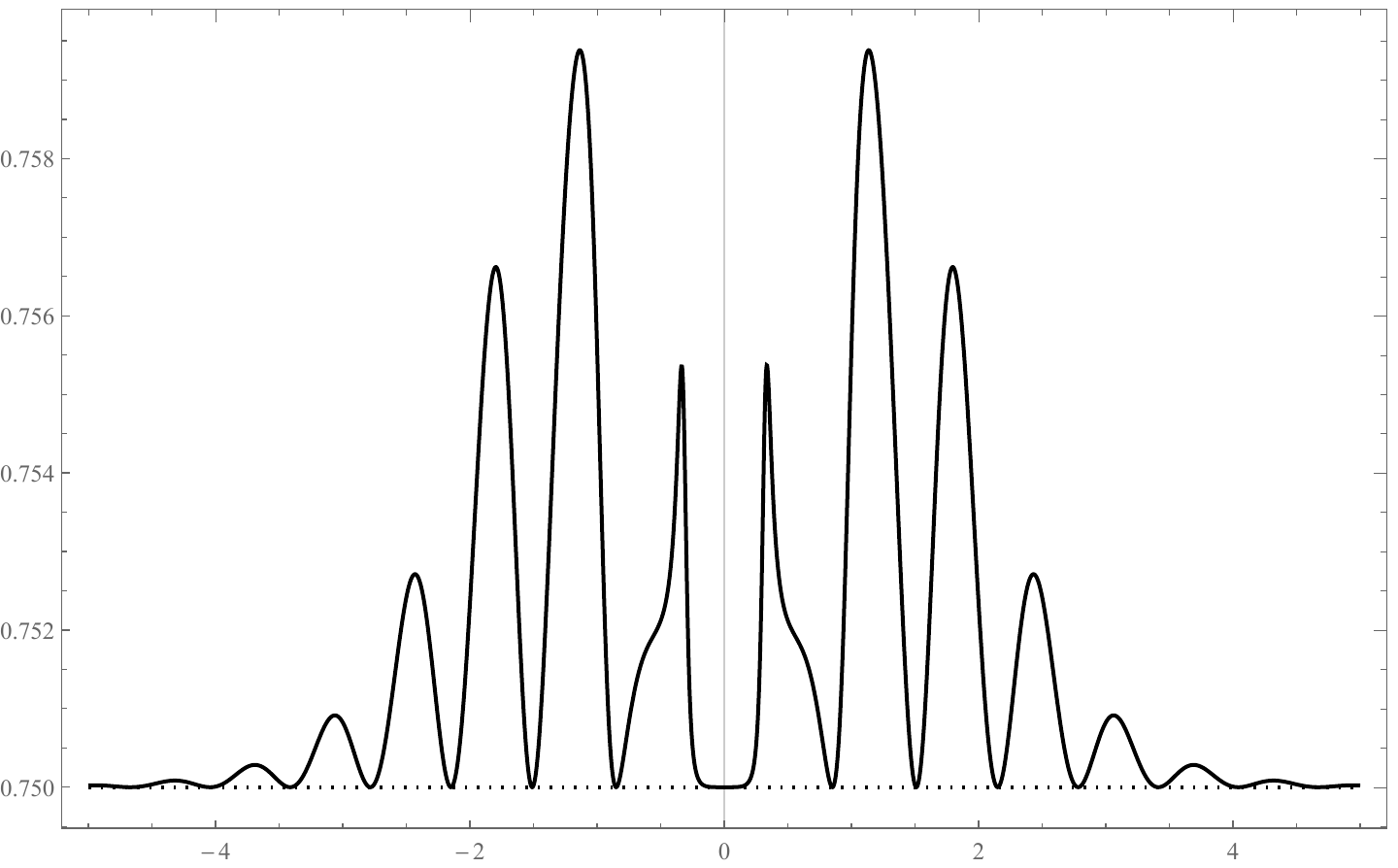}
	\caption{Graph of the function $r_{A_2}$.}
\end{figure}

\subsection{Matrix symmetry is important}\label{SubSym}

The notion of positive definiteness can be understood also in a more general sense, without the symmetry requirement (i.e.~merely that \eqref{pd} holds). However, in this class of functions, inequality \eqref{ide0A} is no longer true, in general. Indeed, for $k\in\N$, consider the non-symmetric matrix
\begin{equation*}
A_3(x)=\left(\begin{matrix}
\cosh{x}&k\\
0&k^2
\end{matrix}\right),\quad x\in\R.
\end{equation*}
Since, by Young's inequality, we have
\begin{align*}
\langle A_3(x)(a,b),(a,b)\rangle_2&=\langle(\cosh x\,a+kb,k^2b),(a,b)\rangle_2=\cosh x\,a^2+kab+k^2b^2\\
&\geq\big(\cosh x-\tfrac12\big)a^2+\tfrac12k^2b^2>0
\end{align*}
for all $a,b\in\R$, the matrix $A_3(x)$ is nonsymmetric positive definite for all $x\in\R$. Then, we compute
\begin{equation*}
r_{A_3}(x)=\frac{(\sinh x)(3\cosh^2 x\sinh x)}{\left|\left(\begin{matrix}
	2\cosh{x}\sinh{x}&k\sinh x\\
	0&0
	\end{matrix}\right)\right|_{2\x2}^2}=\frac{3}{4+\frac{k^2}{\cosh^2x}}<\frac34,
\end{equation*}
which contradicts \eqref{inee}. Moreover, as $k\to\infty$, we have $r_{A_3}(0)\to0$ and thus, there exists no positive multiplicative constant with which \eqref{inee} could hold. Hence, we see that the requirement on the symmetry of $A$ is crucial.

\subsection{Reverse inequality}\label{SecRev}

To give a complete picture about $\eqref{gen}_2$ and its generalization for symmetric positive definite functions, we investigate also the reverse inequality to \eqref{ide0A}. Using an elementary approach, we prove the following result, which however seems optimal.

\begin{theorem}\label{Trev}
	Let $\alpha,\beta\in\R$ and $A\in\Win$. Then
	\begin{equation}\label{ideee}
	\langle\D A^{\alpha},\D A^{\beta}\rangle_{n\x d\x d}\leq\big|\D A^{\frac{\alpha+\beta}2}\big|_{n\x d\x d}^2.
	\end{equation}
\end{theorem}
\begin{proof}
	For any $p,q\in\N$, $B\in\Win$, we use the product rule, \eqref{I2} and \eqref{I1} to write
	\begin{align}
	\big\langle\D B^{2q},\,&\D B^{2p}\big\rangle_{n\x d\x d}=\bigg\langle\sum_{i=0}^{2q-1}B^i(\D B)B^{2q-1-i},\sum_{j=0}^{2p-1}B^j(\D B)B^{2p-1-j}\bigg\rangle_{n\x d\x d}\nonumber\\
	&=\sum_{i=0}^{2q-1}\sum_{j=0}^{2p-1}\big|B^{\frac{i+j}2}(\D B)B^{p+q-1-\frac{i+j}2}\big|^2_{n\x d\x d}\nonumber\\
	&=\sum_{s=0}^{2p+2q-2}\!\!\!\!Q(2q-1,2p-1,s)\big|B^{\frac{s}2}(\D B)B^{p+q-1-\frac{s}2}\big|^2_{n\x d\x d},\label{nec0}
	\end{align}
	where
	\begin{equation*}
	Q(b,a,s)=\min\{b,s\}+\min\{a,s\}-s+1
	\end{equation*}
	is the number of decompositions of the form $v+w=s$ with $v\in\{0,1,\ldots,b\}$ and $w\in\{0,1,\ldots,a\}$. Proceeding completely analogously, we find that
	\begin{equation}\label{nec}
	\big|\D B^{p+q}\big|_{n\x d\x d}^2=\sum_{s=0}^{2p+2q-2}\!\!\!\!Q(p+q-1,p+q-1)\big| B^{\frac s2}(\D B)B^{p+q-1-\frac{s}2}\big|_{n\x d\x d}^2.
	\end{equation}
	Hence, using the simple inequality 
	\begin{equation*}
	Q(2q-1,2p-1,s)\leq Q(p+q-1,p+q-1,s),
	\end{equation*}
	which can be easily verified case by case, we obtain
	\begin{equation*}
	\big\langle\D B^{2q},\D B^{2p}\big\rangle_{n\x d\x d}\leq\big|\D B^{p+q}\big|_{n\x d\x d}^2.
	\end{equation*}
	If we let $E\in\Win$ and choose $B=E^{\frac1{2q}}\in\Win$ (using Lemma~\ref{Lcone} below), this leads to
	\begin{equation*}
	\big\langle\D E,\D E^{\frac pq}\big\rangle_{n\x d\x d}\leq\big|\D E^{(1+\frac pq)/2}\big|_{n\x d\x d}^2,\quad p,q\in\N.
	\end{equation*}
	It can be deduced from \eqref{def} and from \eqref{derd} below that there is a smooth dependence of $\D A^{\lambda}$ on $\lambda>0$, hence
	\begin{equation}\label{eqqE}
	\big\langle\D E, \D E^{\gamma}\big\rangle_{n\x d\x d}\leq\big|\D E^{\frac{1+\gamma}2}\big|_{n\x d\x d}^2,\quad \gamma>0,
	\end{equation}
	by the density of rational numbers in $\R$. Finally, for any $\alpha,\beta\in\R$ such that $\alpha\beta>0$ we choose $\gamma\coloneqq\frac{\alpha}{\beta}>0$ and $E\coloneqq A^{\beta}$ in \eqref{eqqE} to get \eqref{ideee}.
	
	The remaining case $\alpha\beta\leq0$ is trivial since the left hand side of \eqref{ideee} becomes non-positive. Indeed, this can be easily seen if we use \eqref{def}, \eqref{I2} and \eqref{I1}.
\end{proof}

We remark that the same method (i.e.\ expanding the powers as in \eqref{nec0}) could be also used to prove \eqref{ide0A}, however, with a sub-optimal multiplicative constant.

Let us consider the function
\begin{equation}\label{A4}
A_4(x)=\left(\begin{matrix}
1&\sin x\\
\sin x&m
\end{matrix}\right),\quad x\in[-1,1].
\end{equation}
where $m>2$. The matrix $A_4(x)$ is obviously positive definite for all $x\in[-1,1]$ and, recalling the definition of $r_A$ in \eqref{rdef}, we compute that
\begin{align*}
r_{A_4}(x)&=\frac{2\cos^2x(m^2+m+1+3\sin^2x)}{\left|\left(\begin{matrix}
	2\sin x\cos x&(m+1)\cos x\\
	(m+1)\cos x&2\sin x\cos x
	\end{matrix}\right)\right|_{2\x2}}=\frac{m^2+m+1+3\sin^2x}{m^2+2m+1+4\sin^2x},
\end{align*}
hence $r_{A_4}(0)\to1$ as $m\to\infty$. This example indicates that the multiplicative constant in \eqref{ideee} can not be improved, in general. 

Inequality \eqref{ideee} is obviously only a partial converse to \eqref{ide0A} since it misses ``half'' of the left hand side (i.e.\ $q=0$). This omission is necessary as, e.g., the inequality
\begin{equation}\label{zkus}
|\D A^3|_{2\x2}|\D A|_{2\x2}\leq C\big|\D A^2\big|_{2\x2}^2
\end{equation}
with some $C>0$ can not hold in general. To see this, we choose
\begin{equation*}
A_5(x)=\left(\begin{matrix}
	2+\cos x&\sin x\\
	\sin x&m,
\end{matrix}\right),\quad x\in\R,
\end{equation*}
which is a symmetric positive definite matrix for any $m>1$, since
\begin{equation*}
\langle A_5(x)(a,b),(a,b)\rangle_{2}=a^2(2+\cos x)+2ab\sin x+b^2m\geq a^2-2ab+b^2 m>0
\end{equation*}
for all $a,b,x\in\R$. Then, we compute
\begin{equation*}
\frac{|\D A_5(\tfrac{\pi}2)|_{2\x2}|\D A_5^3(\tfrac{\pi}2)|_{2\x2}}{|\D A_5^2(\tfrac{\pi}2)|_{2\x2}^2}=\frac{\sqrt{2m^2+16m+229}}{18},
\end{equation*}
which diverges as $m\to\infty$, violating \eqref{zkus} for all $C>0$.

Our final remark about Theorem~\ref{Trev} concerns the case $\alpha\beta<0$. We may ask if Theorem~\ref{Trev} would still hold in that case if \eqref{ideee} was replaced by an inequality
\begin{equation}\label{fg}
\langle\D^{\alpha} A,\D^{\beta} A\rangle_{n\x d\x d}\leq C \big|\D^{\frac{\alpha+\beta}2} A\big|_{n\x d\x d}^2
\end{equation}
for some $C>0$, where the left hand side now becomes positive. The following example shows that the answer is generally negative. We set $\alpha=1$, $\beta=-1$, $m>2$ and, recalling \eqref{A4}, we evaluate
\begin{equation*}
\frac{\langle\D^1 A_4,\D^{-1}A_4\rangle_{2\x2}}{|\D^0A_4|_{2\x2}^2}=\frac{-\langle\D A_4,\D A_4^{-1}\rangle_{2\x2}}{|\D\log A_4|_{2\x2}^2}=\frac{(m-1)^2}{m(\log m)^2},
\end{equation*}
which diverges as $m\to\infty$, showing that \eqref{fg} can not hold, regardless of how large $C>0$ is. 

All the examples above indicate that Theorem~\ref{Todh} and Theorem~\ref{Trev} may not be improved in any obvious way.

\section{Proofs of the main results}\label{SecProof}

We start by proving Lemma~\ref{Lemf}, which is the cornerstone of our estimates. Note that its conclusion is trivial if the matrices $B$ and $X$ commute (as in the case $d=1$).

\begin{proof}[Proof of Lemma~\ref{Lemf}]
	Let us define the function
	\begin{equation}\label{f}
		f(x)=\Big\langle\int_0^1B^{(1+x)s}XB^{-(1+x)s}\dd{s},\int_0^1B^{(1-x)s}XB^{-(1-x)s}\dd{s}\Big\rangle_{\!\!d\x d},\quad x\in\R.
	\end{equation}	
	Using the formula (which is standard in the ODE theory)
	\begin{equation*}
	\frac{\dd{}}{\dd{x}}\exp(xY)=Y\exp(xY)=\exp(xY) Y,\quad Y\in\R^{d\x d},
	\end{equation*}
	we find, for any $a,b\in\R$ that
	\begin{equation}\label{derd}
	\frac{\dd{}}{\dd{x}} A^{a+bx}=\frac{\dd{}}{\dd{x}}\exp\big((a+bx)\log{A}\big)=b A^{a+bx}\log A=b\log A\,A^{a+bx},
	\end{equation}
	where $A\in\PD$. From this we can deduce that $f$ is a smooth function in $\R$. Moreover, due to the commutativity of the inner product appearing in \eqref{f}, the function $f$ is even (and hence $f'(0)=0$). That $x=0$ is a point of global minimum of $f$ then follows from convexity of $f$, which we now prove by showing that $f''\geq0$ in $\R$.
	
	Using \eqref{I2} and the symmetry of $B$, we can write
	\begin{equation*}
	f(x)=\int_0^1\int_0^1\langle B^{s+t+(s-t)x}XB^{-s-t-(s-t)x},X\rangle_{d\x d}\dd{s}\dd{t}.
	\end{equation*}
	Thus, setting $x_{st}\coloneqq s+t+(s-t)x$, $L\coloneqq\log B$ and using \eqref{derd}, \eqref{I2}, we get
	\begin{align*}
	f'(x)&=\int_0^1\int_0^1(s-t)\langle LB^{x_{st}}XB^{-x_{st}}- B^{x_{st}}XB^{-x_{st}}L,X\rangle_{d\x d}\dd{s}\dd{t}\nonumber\\
	&=\int_0^1\int_0^1(s-t)\langle B^{x_{st}}XB^{-x_{st}},LX-XL\rangle_{d\x d}\dd{s}\dd{t}.
	\end{align*}
	Furthermore, relying on \eqref{I2} and \eqref{I1}, we find that
	\begin{align*}
	f''(x)&=\int_0^1\int_0^1(s-t)^2\langle B^{x_{st}}XB^{-x_{st}},L^2X-LXL\rangle_{d\x d}\dd{s}\dd{t}-\int_0^1\int_0^1(s-t)^2\langle B^{x_{st}}XB^{-x_{st}},LXL-XL^2\rangle_{d\x d}\dd{s}\dd{t}\\
	&=\int_0^1\int_0^1(s-t)^2\Big(\big|B^{\frac12x_{st}}LXB^{-\frac12x_{st}}\big|_{d\x d}^2-2\big\langle B^{\frac12x_{st}}LXB^{-\frac12x_{st}},B^{\frac12x_{st}}XLB^{-\frac12x_{st}}\big\rangle_{d\x d}+\big|B^{\frac12x_{st}}XLB^{-\frac12x_{st}}\big|_{d\x d}^2\Big)\dd{s}\dd{t}\\
	&=\int_0^1\int_0^1\big|(s-t)B^{\frac12x_{st}}(LX-XL)B^{-\frac12x_{st}}\big|^2_{d\x d}\dd{s}\dd{t}
	\end{align*}
	is non-negative for all $x\in\R$, which proves the convexity of $f$ and, consequently, that $f(x)\geq f(0)$ for all $x\in\R$. Rewriting this using \eqref{f} and \eqref{P}, we arrive at \eqref{Odh}.
\end{proof}

In the next proof, we apply Lemma~\ref{Lemf} in its explicit form
\begin{equation}\label{odh}
	\Big\langle\int_0^1\!\!B^{(1+x)s}XB^{-(1+x)s}\dd{s},\int_0^1\!\!B^{(1-x)s}XB^{-(1-x)s}\dd{s}\Big\rangle_{d\x d}\geq\Big|\int_0^1\!\!B^{s}XB^{-s}\dd{s}\Big|_{d\x d}^2.
\end{equation}

\begin{proof}[Proof of Theorem~\ref{TT}]
	Let us first exclude the case $\alpha+\beta=0$. We set $\lambda\coloneqq-\frac{\alpha+\beta}2$, $x\coloneqq\frac{\alpha-\beta}{\alpha+\beta}$ and apply \eqref{def} (proved in Lemma~\ref{Lrepr} below) twice, properties of the matrix power, \eqref{I2} and inequality \eqref{odh} to get
	\begin{align*}
	\big\langle\D^{\alpha}A,\D^{\beta}A\big\rangle_{n\x d\x d}&=\Big\langle\int_0^1\!\!A^{\alpha(1-s)}(\D^0A)A^{\alpha s}\dd{s},\int_0^1\!\!A^{\beta(1-s)}(\D^0A)A^{\beta s}\dd{s}\Big\rangle_{\!\!n\x d\x d}\\
	&=\Big\langle\!\int_0^1\!\!A^{-\alpha s}\big(A^{\frac{\alpha+\beta}2}\D^0A\big)A^{\alpha s}\dd{s},\int_0^1\!\!A^{-\beta s}\big(A^{\frac{\alpha+\beta}2}\D^0A\big)A^{\beta s}\dd{s}\!\Big\rangle_{\!\!n\x d\x d}\\
	&=\Big\langle\int_0^1(A^{\lambda})^{(1+x)s}\big(A^{\frac{\alpha+\beta}2}\D^0A\big)(A^{\lambda})^{-(1+x)s}\dd{s},\int_0^1(A^{\lambda})^{(1-x)s}\big(A^{\frac{\alpha+\beta}2}\D^0A\big)(A^{\lambda})^{-(1-x)s}\dd{s}\Big\rangle_{\!\!n\x d\x d}\\
	&\geq\Big|\int_0^1(A^{\lambda})^{s}\big(A^{\frac{\alpha+\beta}2}\D^0A\big)(A^{\lambda})^{-s}\dd{s}\Big|_{n\x d\x d}^2\\
	&=\Big|\int_0^1A^{\frac{\alpha+\beta}2(1-s)}(\D^0A)A^{\frac{\alpha+\beta}2s}\dd{s}\Big|_{n\x d\x d}^2\\
	&=\big|\D^{\frac{\alpha+\beta}2}A\big|_{n\x d\x d}^2.
	\end{align*}
	
	In the case $\alpha+\beta=0$, we use $|X|_{d\x d}=|X^T|_{d\x d}$ and symmetry of $A$, $\D^0A$, then we apply the Cauchy-Schwarz inequality, \eqref{I2}, \eqref{I1} and get
	\begin{align*}
	\langle&\D^{\alpha}A,\D^{\beta}A\rangle_{n\x d\x d}\\
	&=\Big\langle\int_0^1A^{\alpha(1-s)}(\D^0A)A^{\alpha s}\dd{s},\int_0^1A^{\beta(1-s)}(\D^0A)A^{\beta s}\dd{s}\Big\rangle_{\!\!n\x d\x d}\\
	&=\int_0^1\!\!\int_0^1\big\langle A^{-\alpha s}(\D^0A)A^{\alpha s},A^{-\beta t}(\D^0A)A^{\beta t}\big\rangle_{n\x d\x d}\dd{s}\dd{t}\\
	&=\int_0^1\!\!\int_0^1\big|A^{-\frac12(\alpha s+\beta t)}(\D^0A)A^{\frac12(\alpha s+\beta t)}\big|^2_{n\x d\x d}\dd{s}\dd{t}\\
	&=\int_0^1\!\!\int_0^1\big|A^{-\frac12(\alpha s+\beta t)}(\D^0A)A^{\frac12(\alpha s+\beta t)}\big|_{n\x d\x d}\big|A^{\frac12(\alpha s+\beta t)}(\D^0A)A^{-\frac12(\alpha s+\beta t)}\big|_{n\x d\x d}\dd{s}\dd{t}\\
	&\geq\int_0^1\!\!\int_0^1\big\langle A^{-\frac12(\alpha s+\beta t)}(\D^0A)A^{\frac12(\alpha s+\beta t)}, A^{\frac12(\alpha s+\beta t)}(\D^0A)A^{-\frac12(\alpha s+\beta t)}\big\rangle_{n\x d\x d}\dd{s}\dd{t}\\
	&=|\D^0 A|_{n\x d\x d}^2=\big|\D^{\frac{\alpha+\beta}2}A\big|_{n\x d\x d}^2.
	\end{align*}
	Hence, the property \eqref{cox} follows and Theorem~\ref{TT} is proved.
	\end{proof}	
	
	Up to some auxiliary results, Theorem~\ref{Todh} is an easy consequence of Theorem~\ref{TT}.
	
	\begin{proof}[Proof of Theorem~\ref{Todh}]
		The identity \eqref{gengen} is a direct consequence of \eqref{def} (proved in Lemma~\ref{Lrepr} below) and of \eqref{I2} since
		\begin{align*}
		\langle\D^{\alpha}A, A^{\beta}\rangle_{d\x d}&=\Big\langle\int_0^1A^{\alpha(1-s)}(\D^0A)A^{\alpha s}\dd{s}, A^{\beta}\Big\rangle_{\!\!d\x d}\\
		&=\int_0^1\big\langle A^{\alpha(1-s)}(\D^0A)A^{\alpha s}, A^{\beta(1-s)+\beta s}\big\rangle_{d\x d}\dd{s}\\
		&=\int_0^1\big\langle A^{(\alpha+\beta)(1-s)}(\D^0A)A^{(\alpha+\beta)s},I\big\rangle_{d\x d}\dd{s}=\langle\D^{\alpha+\beta}A,I\rangle_{d\x d}.
		\end{align*}
		
		To prove \eqref{ide0A}, we apply \eqref{cox} of Theorem~\ref{TT} twice and then we use the logarithmic convexity of $\lambda\mapsto|\D^{\lambda}A|_{n\x d\x d}$, which follows from \eqref{cox} (see Lemma~\ref{Llogcon} below):
		\begin{align*}
		\langle\D^{\alpha}A,\D^{\beta}A\rangle_{n\x d\x d}^p\langle\D^{\gamma}A,\D^{\delta}A\rangle_{n\x d\x d}^q
		&\geq\big|\D^{\frac{\alpha+\beta}2}A\big|_{n\x d\x d}^{2p}\big|\D^{\frac{\gamma+\delta}2}A\big|^{2q}_{n\x d\x d}\\
		&=\Big(\big|\D^{\frac{\alpha+\beta}2}A\big|_{n\x d\x d}^{\frac{p}{p+q}}|\D^{\frac{\gamma+\delta}2}A|_{n\x d\x d}^{\frac{q}{p+q}}\Big)^{2p+2q}\\
		&\geq\big|\D^{\frac{\alpha p+\beta p+\gamma q+\delta q}{2p+2q}}A\big|_{n\x d\x d}^{2p+2q}.
		\end{align*}
	\end{proof}

We remark that by iterating the above argument, it is of course possible to include more terms of the same form in the product on the left hand side of \eqref{ide0A}.

Note also that the special case of \eqref{gengen} recovers the Jacobi's formula. Indeed, decomposition \eqref{rozkl} and properties of the matrix determinant, imply that $\langle\,\log A,I\rangle_{d\x d}=\log\det A$. 
This, together with \eqref{gengen} for $\alpha=1$, $\beta=-1$ gives
\begin{equation}\label{JAC}
\langle\D A,A^{-1}\rangle_{d\x d}=\D\log\det A,\quad A\in\Win.
\end{equation}

\section{Auxiliary results}

As we suggested above, the strengthened logarithmic convexity provided by Theorem~\ref{TT} yields the logarithmic convexity for functions of the form $\lambda\mapsto|X(\lambda)|_{n\x d\x d}$.

\begin{lemma}\label{Llogcon}
	Let $H$ be a real vector space with the scalar product $\langle\cdot,\cdot\rangle_H$ and the corresponding norm $|\cdot|_H\coloneqq\sqrt{\langle\cdot,\cdot\rangle_H}$. Let the function $X\colon\R\to H$ be such that $|X|_H$ is Lebesgue-measurable in $\R$ and
	\begin{equation}\label{cox2}
	\big|X(\tfrac{\alpha+\beta}2)\big|^2_H\leq\langle X(\alpha),X(\beta)\rangle_H\quad\text{for all}\quad \alpha,\beta\in\R.
	\end{equation}
	Then $|X|_H$ is logarithmically convex.
\end{lemma}
\begin{proof}
	Let $\varepsilon>0$. If we apply \eqref{cox2}, the Cauchy-Schwarz inequality and the Young inequality in that order, we arrive at
	\begin{align*}
	\big(\varepsilon+\big|X(\tfrac{\alpha+\beta}2)\big|_H\big)^2&\leq\big(\varepsilon+\sqrt{\langle X(\alpha),X(\beta)\rangle_H}\big)^2\leq\big(\varepsilon+|X(\alpha)|^{\frac12}_H|X(\beta)|_H^{\frac12}\big)^2\\
	&=\varepsilon^2+2\varepsilon|X(\alpha)|_H^{\frac12}|X(\beta)|_H^{\frac12}+|X(\alpha)|_H|X(\beta)|_H\\
	&\leq\varepsilon^2+\varepsilon|X(\alpha)|_H+\varepsilon|X(\beta)|_H+|X(\alpha)|_H|X(\beta)|_H\\
	&=\big(\varepsilon+|X(\alpha)|_H\big)\big(\varepsilon+|X(\beta)|_H\big).
	\end{align*}
	Then, taking the logarithm of both sides of this inequality, we obtain
	\begin{equation*}
	\log\big(\varepsilon+\big|X(\tfrac{\alpha+\beta}2)\big|_H\big)\leq\frac12\log\big(\varepsilon+|X(\alpha)|_H\big)+\frac12\log\big(\varepsilon+|X(\beta)|_H\big)
	\end{equation*}
	for all $\alpha,\beta\in\R$, which shows that the real function
	\begin{equation*}
	\ell_{\varepsilon}\colon\lambda\mapsto\log\big(\varepsilon+|X(\lambda)|_H\big)
	\end{equation*}
	is midpoint convex in $\R$. Since the function $\ell_{\varepsilon}$ is a composition of a smooth function with the measurable function $|X|_H$, it is itself measurable. Hence, midpoint convexity of $\ell_{\varepsilon}$ is equivalent to convexity of $\ell_{\varepsilon}$ by the Blumberg-Sierpi\'nski theorem. This gives
	\begin{equation*}
	\log\big(\varepsilon+\big|X\big((1-\lambda)\alpha+\lambda\beta\big)\big|_H\big)\leq(1-\lambda)\log\big(\varepsilon+|X(\alpha)|_H\big)+\lambda\log\big(\varepsilon+|X(\beta)|_H\big),
	\end{equation*}
	which is equivalent to
	\begin{equation*}\varepsilon+\big|X\big((1-\lambda)\alpha+\lambda\beta\big)\big|_H\leq\big(\varepsilon+|X(\alpha)|_H\big)^{1-\lambda}\big(\varepsilon+|X(\beta)|_H\big)^{\lambda}
	\end{equation*}
	and by taking the limit $\varepsilon\to0_+$, we get
	\begin{equation*}
	\big|X\big((1-\lambda)\alpha+\lambda\beta\big)\big|_H\leq|X(\alpha)|_H^{1-\lambda}|X(\beta)|_H^{\lambda}
	\end{equation*}
	for all $\alpha,\beta\in\R$ and $\lambda\in[0,1]$, which is the desired logarithmic convexity of $|X|_H$.
\end{proof}

We further remark that there are functions $X$, for which $|X|_H$ is logarithmically convex, but the property \eqref{cox2} does not hold, indicating that \eqref{cox2} is a rather strong notion of logarithmic convexity for functions of the form $\lambda\mapsto |X(\lambda)|_H$. Indeed, let us consider the function
\begin{equation*}
X(\lambda)=\left(\begin{matrix}
\sinh\lambda&1\\0&0
\end{matrix}\right),\quad\lambda\in\R.
\end{equation*}
Then
\begin{equation*}
|X(\lambda)|_{2\x2}=\sqrt{\sinh^2\lambda+1}=\cosh\lambda
\end{equation*}
is logarithmically convex in $\R$ since $(\log\cosh\lambda)''=(\cosh\lambda)^{-2}>0$, but
\begin{equation*}
\langle X(0),X(2)\rangle_{2\x2}=1<1+\sinh^21=|X(1)|_{2\x2}^2,
\end{equation*}
violating \eqref{cox2}.

To prove the following lemma, we use different representations of the basic matrix functions than those which were introduced by \eqref{pow}. These representations are much more useful from the analytic point of view.

\begin{lemma}\label{Lcone}
	Let $a,b>0$ and $\alpha\in\R$. If $A\in\Win$, then $aA+bB$, $A^{\alpha}\in\Win$ and $\log A\in W^{1,\infty}_{\rm loc}(V;\Sym)$. Furthermore, if $A\in C^1(V;\PD)$, then $aA+bB$, $A^{\alpha}\in C^1(V;\PD)$ and $\log A\in C^1(V;\Sym)$.
\end{lemma}
\begin{proof}
	It is obvious that $aA+bB\in W^{1,\infty}_{\rm loc}(V;\R^{d\x d})$. Moreover, we have
	\begin{equation}\label{eve}
	\langle(aA+bB)v,v\rangle_d=a\langle Av,v\rangle_d+b\langle Bv,v\rangle_d>0\quad\text{for all}\quad 0\neq v\in\R^d
	\end{equation}
	in $V$, and thus $\Win$ is a convex cone. 
	
	Next, we shall prove that $A^{-1}\in\Win$. Note that the positive definiteness of $A$ holds \textit{everywhere} in $V$ since $A$ is continuous in $V$. Hence, we deduce that the eigenvalues of $A$ are all positive everywhere in $V$ and therefore, the matrix inverse $A^{-1}$ exists everywhere in $V$ and it is a positive definite matrix. Moreover, the function $A^{-1}\colon V\to \PD$ obtained hereby is locally bounded. To see this, let us define the function $\rho\colon\PD\to\R_+$ by
	\begin{equation*}
	\rho(B)\coloneqq\min\{\,\lambda:\det(\lambda I-B)=0\,\},\quad B\in\PD.
	\end{equation*}
	It is a well known fact that the spectrum of a matrix depends continuously on its entries (see, e.g., \cite[p.~539]{Horn1990}), and thus $\rho$ is continuous. From this and the continuity of $A$ we deduce that also the composition $\rho\circ A\colon V\to\R_+$ is continuous. Thus, the function $\rho\circ A$ attains its minimum $m_K>0$ on any compact subset $K$ of $V$. Hence, using \eqref{pow}, we can estimate
	\begin{equation*}
	|A^{-1}|_{d\x d}=|QD^{-1}Q^T|_{d\x d}=|D^{-1}|_{d\x d}\leq\frac{d}{\rho\circ A}\leq\frac{d}{m_K}\quad\text{in}\quad K,
	\end{equation*}
	which proves the local boundedness of $A^{-1}$. Hence, the product $A^{-1}(\D A)A^{-1}$ is well defined and locally bounded a.e.\ in $V$. Since, for any $B\in C^{1}(V;\PD)$, we can write
	\begin{equation*}
	\D B^{-1}=B^{-1}B(\D B^{-1})=B^{-1}\D(BB^{-1})-B^{-1}(\D B)B^{-1}=-B^{-1}(\D B)B^{-1},
	\end{equation*}
	we obtain also 
	\begin{equation*}
	\D A^{-1}=-A^{-1}(\D A)A^{-1}\quad \text{a.e.\ in}\quad V
	\end{equation*}
	and $A^{-1}\in\Win$	by a standard approximation argument.
	
	Next, we prove that $\log A\in\Wi(V;\Sym)$. It follows from the properties proved so far that $((1-s)I+sA)^{-1}\in\Win$ for any $s\in[0,1]$. Then, we invoke the well known integral representation of the matrix logarithm
	\begin{equation}\label{smo}
	\log A=\int_0^1((1-s)I+s A)^{-1}(A-I)\dd{s},\quad A\in\PD,
	\end{equation}
	which can be easily verified by using \eqref{rozkl} on the right hand side of \eqref{smo}, evaluating the integrals on the diagonal and finally applying \eqref{pow} (see \cite[p.~269]{Higham} or \cite[Exc.~2.3.9]{Faraut}, cf.~also \cite{Wouk}). Moreover, by applying the derivative to \eqref{smo} (more precisely, by writing \eqref{smo} for the mollification of $A$, applying the derivative and then taking the limit as above), one can deduce that
	\begin{equation}\label{dlog}
	\D\log A=\int_0^1((1-s)I+sA)^{-1}\D A((1-s)I+sA)^{-1}\dd{s}\quad\text{a.e.\ in}\quad V,
	\end{equation}
	cf.\ \cite[(11.10)]{Higham}, from which we readily see that $\log A\in W^{1,\infty}_{\rm loc}(V;\Sym)$.
	
	It remains to deal with the general matrix power $A^{\alpha}$. To this end, we recall that the function $\exp$ can be given by the everywhere convergent matrix power series
	\begin{equation}\label{serie}
	\exp X=\sum_{k=0}^{\infty}\frac1{k!}X^k,\quad X\in\R^{d\x d}.
	\end{equation}
	Then, it is standard to show that $\exp\colon\R^{d\x d}\to\R^{d\x d}$ is a smooth map (cf.~\cite[Sec.~2.1.]{Faraut}) and that it takes $\Sym$ into $\PD$. Hence, by the virtue of the formula
	\begin{equation}\label{loge}
	A^{\alpha}=\exp\log A^{\alpha}=\exp(\alpha \log A),
	\end{equation}
	which follows from \eqref{explog}, we finally conclude $A^{\alpha}\in\Win$.
	
	The proof of Lemma~\ref{Lcone} for $A\in C^1(V;\PD)$ is analogous and is thus omitted.
\end{proof}

Due to Lemma~\ref{Lcone}, the set $\Win$ provides a very convenient setting for our results. Moreover, this setting is advantageous in PDE applications since if a solution of some system is expected to be at least weakly differentiable, it can be always constructed (at least locally) as a limit of some approximating sequence, consisting of Lipschitz continuous functions (constructed, e.g., by~a convolution, by a semi-discretization, by the approximation lemma from \cite{Acerbi} etc.). Frequently, the solution inherits certain properties of the approximating sequence, in particular, inequalities are often preserved by a weak convergence. Then it is enough to apply our results to such approximations.

Our last result concerns the representation formula for $\D A^{\alpha}$ (or $\D^{\alpha}A$) that was stated already in \eqref{def} and used frequently thereafter. In a different context, this formula for $\alpha=1$ can be found in \cite[(11.9)]{Higham}.

\begin{lemma}\label{Lrepr}
	Let $A\in\Win$ and $\alpha\in\R$. Then, the identity 
	\begin{equation}\label{mocnina}
	\D^{\alpha}A=\int_0^1A^{\alpha(1-s)}(\D\log A)A^{\alpha s}\dd{s}
	\end{equation}
	holds almost everywhere in $V$.
\end{lemma}
\begin{proof}
	It is well known (see \cite[(2.1)]{Wilcox} and references therein, cf.~also \cite[(10.15)]{Higham}) that the formula
	\begin{equation}\label{EXP}
	\D\exp X=\int_0^1\exp((1-s)X)\D X\exp(s X)\dd{s}
	\end{equation}
	holds in the classical sense, i.e., for $X\in C^{1}(V;\Sym)$. Moreover, if $A\in C^1(V;\PD)$, Lemma~\ref{Llogcon} yields $\log A\in C^1(V;\Sym)$. Then \eqref{mocnina} follows if we choose $X=\log A^{\alpha}=\alpha\log A$ in \eqref{EXP}, using \eqref{loge}. In the general case $A\in\Win$, we can again approximate $A$ by its convolution $A_{\varepsilon}$, hereby obtaining
	\begin{equation}\label{mocninae}
	\int_V\langle\D^{\alpha}A_{\varepsilon},\Phi\rangle_{d\x d}=\int_V\Big\langle\int_0^1A_{\varepsilon}^{\alpha(1-s)}(\D\log A_{\varepsilon})A_{\varepsilon}^{\alpha s}\dd{s},\Phi\Big\rangle_{d\x d}
	\end{equation}
	for all $\Phi\in C^{\infty}(V;\R^{d\x d})$ compactly supported. Then, since  $\D^{\alpha}A$ and $\D\log A$ are well defined and locally bounded due to Lemma~\ref{Lcone}, it is standard to take the limit in \eqref{mocninae}, obtaining \eqref{mocnina}.
\end{proof}

\section{Concluding remarks}\label{SecLast}

We provided the most basic calculus for locally Lipschitz continuous functions whose codomain is the set of symmetric positive definite matrices. It was shown that, although we need to relieve from the equality sign in $\eqref{gen}_2$, our results are optimal in many aspects. We illustrated that our results apply directly in the theory of tensorial partial differential equations, also due to a~rather mild smoothness assumption $A\in\Win$. Nevertheless, we would like to remark that this assumption can be further relaxed if needed.

Focusing, e.g., on \eqref{ide0A} and replacing the space $\Win$ by $W^{1,r}_{\rm loc}(V;\PD)$ for certain $1\leq r<\infty$, we need to face two additional issues. First, we have to ensure that the left hand side of \eqref{ide0A} is well defined and locally integrable (so that it can be well approximated by smooth functions). When this happens depends crucially on the exponents $\alpha$, $\beta$, $\gamma$, $\delta$, but also on $n$ and $V$ (due to Sobolev embeddings) and the complete characterization would get too complicated. The second issue may occur in the case where some of the exponents $\alpha$, $\beta$, $\gamma$, $\delta$ are negative. Note that $A\in W^{1,r}_{\rm loc}(V;\PD)$ no longer implies continuity of $A$ if $r\leq n$, and $A^{-1}$ may then develop singularities inside $V$ even if $A$ is positive definite in $V$. Hence, in this situation, one has to introduce additional assumptions, such as $A^{-1}\in W^{1,q}_{\rm loc}(V;\PD)$ with appropriately chosen $q$. Then, the idea is to use \eqref{ide0A} for the approximation $(A+\varepsilon I)_{\varepsilon}$ and pass to the limit $\varepsilon\to0_+$. Another obvious remedy is assuming the uniform positive definiteness, i.e., that there exists $\lambda>0$, such that $\langle Av,v\rangle_d\geq\lambda|v|_d^2$ for all $v\in\R^d$. A detailed treatment of these modifications is omitted, as we believe that the setting provided by the space $\Win$ is sufficiently general.

It seems that the proof of Theorem~\ref{Todh} illuminates several interesting mathematical results of a more abstract nature. These results would be difficult to conjecture based only on their scalar version. For example, although the function $\lambda\mapsto\log|\frac1{\lambda}\D A^{\lambda}|_{n\x d\x d}$ is linear if $d=1$, there seems to be no obvious reason, why the same function should be convex if $d>1$ (as claimed in Theorem~\ref{TT}). Next, while it is easy to see that the scalar function $p$ fulfils \eqref{Odh} if and only if the even part of the function $\log p$ attains its global minimum at $0$, such a characterization becomes quite ambiguous in the tensorial case, although the form of the inequality \eqref{Odh} remains the same. Here it seems that the choice of the inner product on $\R^{d\x d}$ plays a prominent role and in our case, the Frobenius inner product is considered as it arises naturally in the PDE applications (cf.\ Section~\ref{SecMot}). Note that Lemma~\ref{Lemf} provides only one example of matrix function (although quite non-trivial) satisfying \eqref{Odh}, while again this example is of no value in the scalar case. It thus seems that there is plenty of room for further exploration.

\end{document}